\documentclass[11pt,noamsfonts]{amsart}
\usepackage{amsthm, amsmath, mathtools}
\allowdisplaybreaks

\usepackage{unicode-math}
\usepackage{microtype}
\usepackage[margin=1in]{geometry}
\usepackage[shortlabels]{enumitem}
\usepackage{graphicx}

\usepackage{tikz}
\usepackage{subcaption}
\newcommand{\solidnodes}{\tikzset{every node/.style={draw,circle,fill=black,minimum size=4pt,inner sep=0pt}}}
\tikzset{textlabel/.style={draw=none,fill=none,shape=rectangle,inner sep=1pt}}
\newcommand{\figscale}{0.85}
\newsavebox{\tikzfigbox}
\newcommand{\tikzfig}[1]{%
  \sbox\tikzfigbox{\scalebox{\figscale}{\input{figures/#1.tikz}}}%
  \ifdim\wd\tikzfigbox>\textwidth
    \GenericWarning{}{Figure #1 is \the\wd\tikzfigbox\space wide, over
      \the\textwidth: lower its scale= key}%
  \fi
  \usebox\tikzfigbox}
\usepackage[backend=biber,style=numeric,sorting=nyt,maxbibnames=99]{biblatex}
\usepackage[colorlinks=true, linkcolor=blue, citecolor=blue, urlcolor=blue]{hyperref}
\hypersetup{
  pdftitle={On some structural properties of graphs with nonnegative resistance curvature},
  pdfauthor={Gyaneshwar Agrahari, Christin Bibby, Sean Boros, Hailey Jay Garcia, Fernando Heidercheidt, Zhiyu Wang},
  pdfsubject={Combinatorics: resistance curvature, toughness, traceability, and Cartesian products of graphs},
  pdfkeywords={resistance curvature, resistance nonnegative, toughness,
    traceability, tree double-matching polytope, sprawling graph, grid graph},
}

\newtheorem{theorem}{Theorem}
\newtheorem{lemma}[theorem]{Lemma}
\newtheorem{corollary}[theorem]{Corollary}
\newtheorem{conjecture}[theorem]{Conjecture}
\newtheorem{proposition}[theorem]{Proposition}
\newtheorem{question}[theorem]{Question}
\usepackage{thm-restate}
\theoremstyle{definition}
\newtheorem{example}[theorem]{Example}

\newcommand{\st}{\:\colon}

\newcommand{\bbR}{\mathbb{R}}
\newcommand{\cT}{\mathcal{T}}
\newcommand{\cS}{\mathcal{S}}
\newcommand{\be}{\mathbf{e}}
\newcommand{\bx}{\mathbf{x}}
\newcommand{\by}{\mathbf{y}}
\DeclareMathOperator{\conv}{conv}
\DeclareMathOperator{\aff}{aff}

\title{On some structural properties of graphs with nonnegative resistance curvature}
\author[Agrahari]{Gyaneshwar Agrahari}
\address{Louisiana State University, Baton Rouge}
\email{gagrah1@lsu.edu}
\author[Bibby]{Christin Bibby}
\address{Louisiana State University, Baton Rouge}
\email{bibby@lsu.edu}
\author[Boros]{Sean Boros}
\address{Louisiana State University, Baton Rouge}
\email{sboros1@lsu.edu}
\author[Garcia]{Hailey Jay Garcia}
\address{Louisiana State University, Baton Rouge}
\email{jgarc86@lsu.edu}
\author[Heidercheidt]{Fernando Heidercheidt}
\address{Louisiana State University, Baton Rouge}
\email{fheide1@lsu.edu}
\author[Wang]{Zhiyu Wang}
\address{Louisiana State University, Baton Rouge}
\email{zhiyuw@lsu.edu}

\subjclass[2020]{Primary 05C75; Secondary 05C05, 05C42, 05C45, 52B05, 52B40}
\keywords{resistance curvature, resistance nonnegative, toughness,
  traceability, tree double-matching polytope, sprawling graph, grid graph}

\begin{document}
\begin{abstract}
A graph is called resistance nonnegative (RN), respectively resistance positive (RP), if it admits positive edge weights such that all vertex resistance curvatures are nonnegative, respectively positive.
In this paper, we study the structure of RN and RP graphs in relation to toughness, traceability, and Cartesian products.
First, we give a characterization of RP in terms of the dimension of the tree double-matching (TDM) polytope.
Next, we disprove a conjecture of Fiedler and answer a question of Devriendt in the negative by constructing, for every $n\ge 9$, an $n$-vertex $1$-tough graph that is not RN.
We then show that RP graphs need not be traceable by proving that the Thomassen $34$-graph is RP but not traceable.
Finally, we resolve a conjecture of Devriendt on grid graphs by proving that all Cartesian products of paths are RN.
\end{abstract}

\maketitle

\section{Introduction}\label{sec:intro}

Ricci curvature plays an important role in the geometric analysis of Riemannian manifolds.
In the smooth setting, Ricci curvature is obtained by tracing the sectional curvature, while scalar curvature is the trace of the Ricci tensor.
Thus Ricci curvature records directional volume distortion, whereas scalar curvature assigns a single curvature value to each point.
Many variants of Ricci curvature have been introduced in nonsmooth and discrete spaces, e.g., the Bakry--\'Emery lower Ricci curvature bound~\cite{Bakry-Emery1985}, Ollivier's coarse Ricci curvature on metric spaces~\cite{Ollivier}, and the synthetic lower Ricci curvature bounds of Lott--Villani and Sturm on metric measure spaces~\cite{Lott-Villani,Sturm}.
In the setting of graphs, Chung and Yau introduced a notion of Ricci-flatness~\cite{Chung-Yau96}.
Later, Lin and Yau~\cite{LY} studied lower Ricci curvature bounds for graphs, and Lin, Lu, and Yau~\cite{LLY} modified Ollivier's Ricci curvature to obtain a graph curvature independent of the idleness parameter of the random walk.
Other important notions of graph curvature include Forman curvature~\cite{Forman2003}, Bakry--\'Emery curvature on graphs~\cite{Schmuckenschlaeger,FM-L}, and combinatorial curvature for planar graphs~\cite{Higuchi2001}.
See, for example,~\cite{Cho-Paeng2013,CKKLP2021,Lu-Wang2026,Paeng2012} and the references therein for further connections between these curvature notions and graph-theoretic properties.

Recently, Devriendt and Lambiotte introduced a new notion of discrete curvature based on effective resistance~\cite{Devriendt-Lambiotte2022}.
Effective resistance originates from electrical network theory and random walks, and was introduced as a graph distance by Klein and Randi\'c~\cite{Klein-Randic1993}; see also~\cite{Doyle-Snell1984,Grimmett2010}.
Unlike shortest-path distance, effective resistance takes into account all paths connecting two vertices and therefore measures how well two vertices are connected through the whole graph.
This global nature makes effective resistance a useful geometric quantity in graph theory, spectral graph theory, and network analysis.

Let $G=(V,E)$ be a finite simple connected graph.
For convenience, given an edge $e\in E$ and a subgraph $H$ of $G$, we use $e\in H$ to denote that $e\in E(H)$.
Denote the number of connected components of $G$ by $c(G)$.
Given positive edge weights $c=(c_e)_{e\in E}$, let $\cT(G)$ denote the set of spanning trees of $G$, and let $\mu_c$ denote the \emph{log-linear} distribution on $\cT(G)$, under which a spanning tree $T$ has probability proportional to $\prod_{h\in T}c_h$.

For an edge $e\in E$, the \emph{relative resistance} of $e$ is the probability that a random spanning tree drawn from $\mu_c$ contains $e$, that is,
\[
   r_e(c)
   :=
   \Pr_{\mu_c}(e\in T)
   =
   \frac{\displaystyle\sum_{\substack{T\in \cT(G)\\ e\in T}} \prod_{h\in T}c_h}{\displaystyle\sum_{T\in \cT(G)} \prod_{h\in T}c_h}.
\]
Equivalently, $r_e(c)=c_e\omega_e(c)$, where $\omega_e(c)$ is the effective resistance between the endpoints of $e$ in $G$ given the edge weights $c$~\cite{Devriendt-Lambiotte2022}; this identity is what gives $r_e(c)$ its name.
Foster's Theorem~\cite{Foster1949}, a classical result in electrical circuit analysis, states that $\sum_{e\in E} r_e(c)=|V|-1$.
The \emph{resistance curvature} of a vertex $v\in V(G)$, denoted by $p_v(c)$, is defined by
\[
   p_v(c)=1-\frac12\sum_{e\ni v}r_e(c).
\]

Following Devriendt~\cite{Devriendt2025}, we call a graph $G$ \emph{resistance nonnegative}, abbreviated RN, if there exists a choice of positive edge weights $c$ such that $p_v(c)\ge 0$ for every $v\in V(G)$.
Similarly, we call $G$ \emph{resistance positive}, abbreviated RP, if there exists a choice of positive edge weights $c$ such that $p_v(c)>0$ for every $v\in V(G)$.
A graph that is RN but not RP is called \emph{strictly resistance nonnegative}, or SRN.

In~\cite{Devriendt2025}, Devriendt characterized RN and RP graphs in terms of distributions on spanning trees, as well as spanning tree and matching polytopes.
For a subset $F\subseteq E(G)$, let $\be_F\in \bbR^{E(G)}$ denote its \emph{indicator vector}, i.e., $(\be_F)_h = 1$ if $h\in F$ and $0$ otherwise.
The \emph{spanning tree polytope} of $G$ is defined by $P(G):=\conv\{\be_T\st T\in \cT(G)\}\subseteq \bbR^{E(G)}$.
We write $P(G)^\circ$ for the relative interior of $P(G)$ in its affine hull.
A \emph{matching} of $G$ is a set of pairwise disjoint edges, and the \emph{matching polytope} of $G$ is defined by $M(G):=\conv\{\be_M\st M \text{ is a matching of }G\} \subseteq \bbR^{E(G)}$.
Let $2M(G):=\{2\bx\st \bx\in M(G)\}$.

A distribution on $\cT(G)$ is called \emph{positive} if every spanning tree receives positive probability, and it is called \emph{nonseparable} if its support is not contained in a proper face of $P(G)$.
Devriendt~\cite{Devriendt2025} gave the following characterization of RN and RP graphs, among others.

\begin{theorem}[Devriendt~\cite{Devriendt2025}]
\label{thm:devriendt-characterization}
Let $G$ be a graph.
Then the following statements hold.
\begin{enumerate}
   \item $G$ is RN, respectively RP, if and only if there exists a nonseparable, log-linear, or positive distribution on $\cT(G)$ such that the expected degree of every vertex in a random spanning tree under this distribution is at most $2$, respectively strictly smaller than $2$.
   Any one of these three types of distributions is sufficient, and all three types are necessary.

   \item $G$ is RN if and only if $P(G)^\circ\cap 2M(G)\neq \emptyset$.
\end{enumerate}
\end{theorem}

In this paper, we first give an alternative characterization of RP graphs in terms of the dimension of the \emph{tree double-matching polytope} $\Theta(G):=P(G)\cap 2M(G)$.

\begin{restatable}{theorem}{dimchar}\label{thm:dimension}
Let $G=(V,E)$ be a connected graph with $|V|\ge 3$.
Then $G$ is RP if and only if $\dim \Theta(G)=|E|-1$.
\end{restatable}

A \emph{Hamiltonian path} or \emph{Hamiltonian cycle} of a graph is a path or cycle, respectively, that visits every vertex exactly once.
A graph is called \emph{Hamiltonian} if it contains a Hamiltonian cycle.
Devriendt~\cite{Devriendt2025} showed that the integer points of the polytope $\Theta(G)$ are precisely the indicator vectors of Hamiltonian paths of $G$.
Theorem~\ref{thm:devriendt-characterization} connects resistance curvature with matchings, spanning trees, and Hamiltonian paths/cycles.
In particular, it suggests that RN and RP graphs are closely related to Hamiltonian-type properties, but the relationship is subtle.

Graph toughness was introduced by Chv\'atal in connection with Hamiltonian cycles~\cite{Chvatal1973}.
A graph $G$ is called \emph{$t$-tough} if, for every vertex set $S\subseteq V(G)$ whose removal disconnects $G$, the number of connected components of $G-S$ is at most $|S|/t$.
Chv\'atal~\cite{Chvatal1973} conjectured that there exists a constant $t_0$ such that every $t_0$-tough graph is Hamiltonian.
It is known that $1$-toughness does not imply Hamiltonicity; see~\cite{Bauer-Broersma-Schmeichel2006}.

In the context of resistance curvature, Devriendt~\cite{Devriendt2025} showed that every Hamiltonian graph is RP.
A result of Fiedler~\cite[Theorem~3.4.18]{Fiedler2011} (see also~\cite[Theorem~6.31]{Devriendt2022}) implies that every connected weighted graph with positive resistance curvature is $1$-tough.
Thus we have
\[
   \{\text{Hamiltonian graphs}\}
   \subseteq
   \{\text{RP graphs}\}
   \subseteq
   \{\text{$1$-tough graphs}\}.
\]
The first containment is strict, as the Petersen graph is RP but non-Hamiltonian.
This led Devriendt~\cite{Devriendt2025} to ask whether every $1$-tough graph is RP; the same statement had been conjectured by Fiedler~\cite{Fiedler2011}.

\begin{question}[{\cite[Question~6.6]{Devriendt2025}}]\label{q:1-tough-RP}
Is every $1$-tough graph RP?
\end{question}

We answer Question~\ref{q:1-tough-RP} in the negative in a stronger form.

\begin{restatable}{theorem}{onetoughNotRP}
\label{thm:1-tough-not-RP}
For every $n\ge 9$, there exists an $n$-vertex $1$-tough non-RN (thus also non-RP) graph.
\end{restatable}

Our next result separates RP from an even weaker Hamiltonian-type property.
A graph is called \emph{traceable} if it contains a Hamiltonian path.
Recall that the Petersen graph is an example demonstrating that RP graphs do not have to be Hamiltonian.
But is it at least true that all RP graphs are traceable?
We answer this question in the negative using the Thomassen $34$-graph $T_{34}$ (Figure~\ref{fig:T34}).

\begin{restatable}{theorem}{thomassenRPNTraceable}
\label{thm:RP-nontraceable}
The Thomassen $34$-graph $T_{34}$ is RP but not traceable.
\end{restatable}

In polyhedral terms, Theorem~\ref{thm:RP-nontraceable} shows that the existence of a curvature-positive point in the relevant spanning-tree/matching polytope need not force the existence of an integer point corresponding to a Hamiltonian path.

For graphs $G$ and $H$, let $G\times H$ denote their Cartesian product, i.e., the graph with vertex set $V(G)\times V(H)$, where two vertices $(g,h)$ and $(g',h')$ are adjacent if and only if either $g=g'$ and $hh'\in E(H)$, or $h=h'$ and $gg'\in E(G)$.
A \emph{grid graph} is a graph of the form $P_n\times P_m$, where $P_k$ is the path on $k$ vertices.
Devriendt~\cite{Devriendt2025} showed that $P_n\times P_m$ is RP whenever $mn$ is even and $m,n>1$, since in this case the grid graph is Hamiltonian.
He further conjectured that all grid graphs are RN~\cite{Devriendt2025} and asked whether Cartesian products of more than two paths are RN.

\begin{conjecture}[{\cite[Conjecture~6.8]{Devriendt2025}}]\label{conj:grid-graphs}
The graph $P_n \times P_m$ is RN for all positive integers $m,n$.
\end{conjecture}

We confirm Conjecture~\ref{conj:grid-graphs}.

\begin{restatable}{theorem}{gridsprawlingRN}\label{thm:grid-RN}
For all positive integers $d,n_1,\ldots,n_d$, the graph $P_{n_1}\times\cdots\times P_{n_d}$ is RN.
\end{restatable}

To show Theorem~\ref{thm:grid-RN}, we develop a more general framework using what we call a \emph{sprawling graph}.
A graph $G=(V,E)$ is called \emph{sprawling} if there exists a collection $\cS=\{H_1,\ldots,H_q\}$ of Hamiltonian paths of $G$ satisfying the following two conditions:
\begin{enumerate}
    \item every edge of $G$ is contained in at least one path in $\cS$;
    \item for every set $U\subsetneq V$ with $2\le |U|\le |V|-1$ such that the subgraph $G[U]$ induced by $U$ is connected, there exists $H_i\in\cS$ such that $H_i[U]$ is not a spanning tree of $G[U]$.
\end{enumerate}
We call such a collection $\cS$ a \emph{sprawling set}.
The next result shows that sprawling graphs give a sufficient condition for RN.

\begin{restatable}{theorem}{sprawlingRN}
\label{thm:sprawling-RN}
Every sprawling graph $G$ is RN.
Moreover, if $|V(G)|\ge 3$, then $G$ is $2$-connected.
\end{restatable}

We then prove Theorem~\ref{thm:grid-RN} by showing that $P_{n_1}\times\cdots\times P_{n_d}$ is sprawling 
for $d\ge 2$ and $n_1,\ldots,n_d\ge 2$.

\subsection*{Organization}
In Section~\ref{sec:tdm}, we provide preliminary background on the tree double-matching polytope and give an alternative characterization of RP graphs in terms of this polytope.
We prove Theorem~\ref{thm:1-tough-not-RP} in Section~\ref{sec:1-tough-RP} and Theorem~\ref{thm:RP-nontraceable} in Section~\ref{sec:RP-nontraceable}.
Finally, in Section~\ref{sec:sprawling}, we discuss structural properties of sprawling graphs and prove Theorems~\ref{thm:grid-RN}and~\ref{thm:sprawling-RN}.

In Appendix~\ref{sec:zoo}, we provide a complete comparison of all relevant graph classes and, in particular, depict containment relationships in Figure~\ref{fig:containment}.

\section{The Tree Double-Matching Polytope}\label{sec:tdm}

In this paper, all graphs are simple.
We follow the notation and terminology of Diestel~\cite{Diestel}.
Let $G=(V,E)$ be a graph.
For a vertex $v\in V$, we denote the degree of $v$ in $G$ by $\deg_G(v)$; when the underlying graph is clear from the context, we omit the subscript.
We say that a connected graph $G$ with $|V|\ge 3$ is \emph{$2$-connected} if $G-v$ is connected for every $v\in V$.
Given a graph $G$ and a vertex subset $U\subseteq V$, let $G[U]$ denote the subgraph of $G$ induced by $U$.
We also recall some basic terminology concerning polytopes.
For points $\bx_1,\ldots,\bx_n\in \bbR^k$, their \emph{convex hull} is
\[
   \conv(\bx_1,\ldots,\bx_n)
   :=
   \left\{
      \sum_{i\in [n]}\lambda_i\bx_i
      \st
      \lambda_i\in [0,1]\text{ for all }i\in[n],
      \text{ and }\sum_{i\in[n]}\lambda_i=1
   \right\}.
\]
A set of this form is called a \emph{polytope}.
We denote by $\dim P$ the dimension of a polytope $P$, namely the dimension of the smallest affine subspace of $\bbR^k$ containing $P$.
The \emph{relative interior} of a polytope $P$, denoted by $P^\circ$, is the interior of $P$ taken inside its affine hull.
Equivalently, if $\bx_1,\ldots,\bx_n$ are the vertices of $P$, then every point of $P^\circ$ can be written as a convex combination of the $\bx_i$ with all coefficients positive, that is,
\[
   P^\circ
   =
   \left\{
      \sum_{i\in [n]}\lambda_i\bx_i
      \st
      \lambda_i >0\text{ for all }i\in[n],
      \text{ and }\sum_{i\in[n]}\lambda_i=1
   \right\}.
\]
We will mostly consider polytopes in the ambient space $\bbR^E$, whose coordinates are indexed by the edges of the graph $G$.

We take the spanning tree polytope $P(G)$, the double-matching polytope $2M(G)$, and the TDM polytope $\Theta(G)$ as defined in Section~\ref{sec:intro}.
In order to study $\Theta(G)$, we first recall standard descriptions of the double-matching polytope and the spanning tree polytope.
For $v\in V(G)$, let $E(v)$ denote the set of edges incident with $v$.
For $S\subseteq V(G)$, let $E[S]$ denote the set of edges with both endpoints in $S$.
For $F\subseteq E(G)$ and $\bx\in \bbR^{E(G)}$, write $\bx(F):=\sum_{e\in F}\bx_e$.

\begin{lemma}[{\cite[Theorem~7.3.1]{LP86}}]
\label{lem:matching-polytope}
The double-matching polytope $2M(G)$ for a graph $G=(V,E)$ is the set of points $\bx=(\bx_e)_{e\in E}\in\bbR^{E}$ satisfying
\begin{enumerate}[label=(\roman*)]
    \item $\bx_e\ge 0$ for all $e\in E$;
    \item $\bx(E(v))\le 2$ for all $v\in V$; and
    \item $\bx(E[S])\le |S|-1$ for every odd subset $S\subseteq V$.
\end{enumerate}
\end{lemma}

\begin{lemma}[{\cite[Proposition~2.3]{FS05}}]
\label{lem:treepoly-char}
The spanning tree polytope $P(G)$ for a graph $G=(V,E)$ is the set of points $\bx=(\bx_e)_{e\in E}\in\bbR^{E}$ satisfying
\begin{enumerate}[label=(\roman*)]
    \item $\bx_e\ge 0$ for all $e\in E$;
    \item $\mathbf 1\cdot \bx=|V|-1$, where $\mathbf 1$ is the vector of all ones; and
    \item $\bx(E[S])\le |S|-1$ for every nonempty proper subset $S\subsetneq V$.
\end{enumerate}
\end{lemma}

It is routine to obtain the following description of the relative interior of the spanning tree polytope from Lemma~\ref{lem:treepoly-char}.
We include a short proof for completeness.

\begin{lemma}[Relative interior of the spanning tree polytope]
\label{lem:treepoly-relint}
Let $G=(V,E)$ be a $2$-connected graph.
Then $P(G)^\circ$ is the set of points $\bx=(\bx_e)_{e\in E}\in\bbR^E$ satisfying
\begin{enumerate}[label=(\roman*)]
    \item $\bx_e>0$ for every $e\in E$;
    \item $\mathbf 1\cdot \bx=|V|-1$; and
    \item $\bx(E[S])<|S|-1$ for every $S\subsetneq V$ with $2\le |S|\le |V|-1$.
\end{enumerate}
\end{lemma}
\begin{proof}
We use the standard fact that, if a polytope is described inside its affine hull by finitely many valid inequalities, then its relative interior consists of the points strictly satisfying every inequality that is not an equality on the whole polytope.

By Lemma~\ref{lem:treepoly-char}, the spanning tree polytope $P(G)$ is described by
\[
   \bx_e\ge 0\quad(e\in E), \qquad
   \mathbf 1\cdot \bx=|V|-1, \qquad
   \bx(E[S])\le |S|-1\quad(\emptyset\neq S\subsetneq V).
\]
The singleton inequalities are identities, since $E[\{v\}]=\emptyset$.

We now show that no the remaining inequality is an equality on all of $P(G)$.
First, every edge of $G$ is contained in some spanning tree, so no inequality $\bx_e\ge 0$ is tight on all of $P(G)$.
Now let $S\subsetneq V$ with $2\le |S|\le |V|-1$.
If $G[S]$ is disconnected, then, for every spanning tree $T$ of $G$, the forest $T[S]$ has at most $|S|-2$ edges.
Hence $\bx(E[S])\le |S|-1$ is not tight on all of $P(G)$.
It remains to consider the case where $G[S]$ is connected.
Since $G$ is $2$-connected, every component of $G-S$ has at least two neighbors in $S$.
Choose a component $C$ of $G-S$, and choose two distinct vertices $u,v\in S$ adjacent to $C$.
Let $T_S$ be a spanning tree of $G[S]$, and for each component of $G-S$, choose a spanning tree of that component.
Add one edge from each component of $G-S$ to $S$, and for the chosen component $C$, add two such edges, one incident with $u$ and one incident with $v$.
The resulting connected spanning subgraph has exactly one cycle, and this cycle contains the path in $T_S$ from $u$ to $v$.
Deleting an edge of this path gives a spanning tree $T$ of $G$ such that $T[S]$ is disconnected.
Therefore $|E(T[S])|\le |S|-2$, so the inequality $\bx(E[S])\le |S|-1$ is not tight on all of $P(G)$.

Thus the relative interior of $P(G)$ is obtained from the spanning tree polytope description by making precisely the nontrivial inequalities strict.
This gives the stated description.
\end{proof}

The following description follows immediately from Lemmas~\ref{lem:matching-polytope} and~\ref{lem:treepoly-char}.
We include a short proof for completeness.

\begin{proposition}[Half-space description of the TDM polytope]
\label{prop:tdm-char}
The TDM polytope $\Theta(G)$ for a graph $G=(V,E)$ is the set of points $\bx=(\bx_e)_{e\in E}\in\bbR^{E}$ satisfying
\begin{enumerate}[label=(\roman*)]
    \item \label{ineqi} $\bx_e\ge 0$ for all $e\in E$, and $\mathbf 1\cdot \bx=|V|-1$;

    \item \label{ineqii} $\bx(E(v))\le 2$ for all $v\in V$; and

    \item \label{ineqiii} $\bx(E[S])\le |S|-1$ for every nonempty proper set $S\subsetneq V$.
\end{enumerate}
\end{proposition}

\begin{proof}
Let $Q$ denote the set of points satisfying \ref{ineqi}--\ref{ineqiii}.
If $\bx\in\Theta(G)=P(G)\cap 2M(G)$, then $\bx$ satisfies \ref{ineqi} and \ref{ineqiii} by Lemma~\ref{lem:treepoly-char}, and it satisfies \ref{ineqii} by Lemma~\ref{lem:matching-polytope}.
Hence $\Theta(G)\subseteq Q$.

Conversely, let $\bx\in Q$.
By Lemma~\ref{lem:treepoly-char}, we have $\bx\in P(G)$.
It remains to show that $\bx\in 2M(G)$.
The nonnegativity constraints and vertex inequalities in Lemma~\ref{lem:matching-polytope} are exactly the corresponding parts of \ref{ineqi} and \ref{ineqii}.
Now let $S\subseteq V(G)$ be an odd set.
If $S\subsetneq V(G)$, then \ref{ineqiii} gives $\bx(E[S])\le |S|-1$.
If $S=V(G)$, then \ref{ineqi} gives $\bx(E[S])=\mathbf 1\cdot \bx=|V(G)|-1=|S|-1$.
Thus all odd-set inequalities in Lemma~\ref{lem:matching-polytope} hold, and hence $\bx\in 2M(G)$.
Therefore $Q\subseteq \Theta(G)$.
\end{proof}

This representation is not necessarily minimal.
Some vertex inequalities may be redundant, and the spanning tree inequalities may combine with the vertex inequalities to force additional equalities, thereby lowering the dimension of $\Theta(G)$.
Combining Lemma~\ref{lem:treepoly-relint} with Lemma~\ref{lem:matching-polytope} gives the following description of $P(G)^\circ\cap 2M(G)$.
We omit the proof.

\begin{proposition}[Half-space description of $P(G)^\circ\cap 2M(G)$]
\label{prop:open-tdm-char}
Let $G=(V,E)$ be a $2$-connected graph.
Then $P(G)^\circ\cap 2M(G)$ is the set of points $\bx=(\bx_e)_{e\in E}\in\bbR^E$ satisfying
\begin{enumerate}[label=(\roman*), ref=(\roman*)]
    \item \label{relintineqi} $\bx_e>0$ for all $e\in E$, and $\mathbf 1\cdot \bx=|V|-1$;

    \item \label{relintineqii} $\bx(E(v))\le 2$ for all $v\in V$; and

    \item \label{relintineqiii} $\bx(E[S])<|S|-1$ for every $S\subsetneq V$ with $2\le |S|\le |V|-1$.
\end{enumerate}
\end{proposition}

We will also use the standard fact that, for a connected graph $G=(V,E)$ with $|V|\ge 3$, the spanning tree polytope has dimension $|E|-1$ if and only if $G$ is $2$-connected.
We are now ready to prove Theorem~\ref{thm:dimension}, which gives another characterization of RP graphs in terms of the dimension of the TDM polytope.
We restate Theorem~\ref{thm:dimension} here for convenience.

\dimchar*

\begin{proof}
Let $A=\aff(P(G))$ be the affine hull of the spanning tree polytope, and let
\[
   C:=\{\bx\in\bbR^E\st \bx(E(v))\le 2 \text{ for all } v\in V\}.
\]
By Proposition~\ref{prop:tdm-char}, we have $\Theta(G)=P(G)\cap C$.
By Theorem~\ref{thm:devriendt-characterization}, the graph $G$ is RP if and only if there exists $\bx\in P(G)^\circ$ such that $\bx(E(v))<2$ for every $v\in V$.

Suppose first that $G$ is RP.
Then there is a point $\bx\in P(G)^\circ$ such that $\bx(E(v))<2$ for every $v\in V$.
Since $P(G)^\circ$ is open in $A$ and the inequalities $\bx(E(v))<2$ are strict, a small relative neighborhood of $\bx$ in $A$ is contained in $\Theta(G)$.
Hence $\Theta(G)$ has nonempty relative interior in $A$, and therefore $\dim\Theta(G)=\dim P(G)$.
The only connected RN graphs that are not $2$-connected are the paths $P_n$, of which all are SRN besides the trivial case $K_2$~\cite[Proposition~3.7]{Devriendt2025}.
Therefore, all connected RP graphs with $|V|\ge 3$ are $2$-connected and the spanning tree polytope has dimension $\dim P(G)=|E|-1$.
Thus $\dim\Theta(G)=|E|-1$.

Conversely, suppose that $\dim\Theta(G)=|E|-1$.
Since $\Theta(G)\subseteq P(G)$ and $P(G)$ is contained in the hyperplane $\mathbf 1\cdot \bx=|V|-1$, we have $\dim\Theta(G)\le \dim P(G)\le |E|-1$.
Thus $\dim\Theta(G)=\dim P(G)=|E|-1$.
Hence $\Theta(G)$ has nonempty relative interior inside $A$.
Choose $\bx$ in this relative interior.
Since $\Theta(G)\subseteq P(G)$ and the two polytopes have the same dimension, we have $\bx\in P(G)^\circ$.

We claim that $\bx(E(v))<2$ for every $v\in V$.
Suppose instead that $\bx(E(v))=2$ for some vertex $v$.
Since $\dim P(G)=|E|-1$, the graph $G$ is $2$-connected, and hence $A=\{\by\in\bbR^E\st \mathbf 1\cdot \by=|V|-1\}$.
Because $|V|\ge 3$, the graph $G-v$ is connected and has at least two vertices, so there exists an edge $f$ not incident with $v$.
Choose an edge $e$ incident with $v$.
For sufficiently small $\varepsilon>0$, the point $\by:=\bx+\varepsilon \be_e-\varepsilon \be_f$ still lies in $A$, but satisfies
\[
   \by(E(v))=\bx(E(v))+\varepsilon=2+\varepsilon>2.
\]
Thus $\by\notin\Theta(G)$.
This contradicts the fact that $\bx$ is in the relative interior of $\Theta(G)$ inside $A$: every sufficiently small relative neighborhood of $\bx$ in $A$ must be contained in $\Theta(G)$.
Therefore $\bx(E(v))<2$ for every $v\in V$.
By Theorem~\ref{thm:devriendt-characterization}, $G$ is RP.
\end{proof}

Note that Theorem~\ref{thm:dimension} should not be interpreted as saying that the relative interiors of $P(G)$ and $2M(G)$ intersect.
For example, when $G$ is $2$-connected and \emph{factor-critical} (i.e., $G-v$ has a perfect matching for every $v\in V(G)$), the odd-set inequality $\bx(E)\le |V|-1$ constrains $2M(G)$, while $P(G)$ is contained in the hyperplane $\bx(E)=|V|-1$.
Thus any intersection with $P(G)$ lies on the boundary of $2M(G)$, even though $\Theta(G)$ may still have full dimension inside $\aff(P(G))$.

\section{Existence of \texorpdfstring{$1$}{1}-Tough Non-RN Graphs}\label{sec:1-tough-RP}

To prove Theorem~\ref{thm:1-tough-not-RP}, we construct a graph with a degree $t$ vertex whose neighbors all have degree $2$, as follows.
Let $t\ge 3$, and let the complete graph $K_{t+1}$ have vertex set $\{v_0,v_1,\ldots,v_t\}$.
For positive integers $s_1,\ldots,s_t$, let $G_t(s_1,\ldots,s_t)$ be the graph obtained from $K_{t+1}$ by subdividing the edge $v_0v_i$ exactly $s_i$ times for each $i\in[t]$.
Equivalently, $G_t(s_1,\ldots,s_t)$ is obtained from $K_{t+1}$ by replacing each edge $v_0v_i$ with a path $Q_i$ of length $s_i+1$ from $v_0$ to $v_i$, where the paths $Q_1,\ldots,Q_t$ are internally disjoint; see Figure~\ref{fig:gt-construction}.

\begin{figure}[htbp]
  \centering
  %% G_5(s_1,...,s_5)
\begin{tikzpicture}[scale=1.54404, yscale=.55]
  \solidnodes

  % the pointy bit
  \node (v0) at (3,0) {};
  \node[textlabel,xshift=4pt,yshift=4pt] at (3,-0.8) {$v_0$};

  % v_i, its subdivided path to v_0, and the path's far endpoint x_i.
  \foreach \i/\y/\vgap/\xgap/\xdown in
      {1/3/3/3/0, 2/1.5/3/3/-3, 3/0/3/3/-3, 4/-1.5/5/3/-3, 5/-3/5/3/-3} {
    \node[label={[label distance=\vgap pt]north east:$v_{\i}$}] (v\i) at (0.0,\y) {};
    \node (a\i) at (1,0.9*\y) {};
    \node (c\i) at (1.5,0.75*\y) {};
    \node[label={[label distance=\xgap pt,yshift=\xdown pt]110:$u_{\i}$}] (x\i)
      at (2,0.6*\y) {};

    \draw[thick] (v\i)--(a\i);
    \draw[thick,dotted] (a\i)--(c\i)--(x\i);
    \draw[thick] (x\i)--(v0);
  }

  % K_5
  \draw[thick] (v1)--(v2)--(v3)--(v4)--(v5);
  \draw[thick] (v1) to[bend right=18] (v3);
  \draw[thick] (v2) to[bend right=18] (v4);
  \draw[thick] (v3) to[bend right=18] (v5);
  \draw[thick] (v1) to[bend right=32] (v4);
  \draw[thick] (v2) to[bend right=32] (v5);
  \draw[thick] (v1) to[bend right=46] (v5);
\end{tikzpicture}
  \caption{$G_5(s_1, s_2, s_3, s_4, s_5)$.}
  \label{fig:gt-construction}
\end{figure}

Dawes and Rodrigues~\cite{DawesRodrigues1990} showed that $G_t(s_1,\ldots,s_t)$ is $1$-tough under the stronger assumption that $s_i\ge 3$ for every $i\in[t]$.
The same argument in fact shows that the conclusion holds whenever $s_i\ge 1$ for every $i\in[t]$.
Since their proof is rather compressed, we include a more detailed proof for completeness.

\begin{lemma}\label{lem:Gt-1-tough}
For $t\ge 3$ and positive integers $s_1,\ldots,s_t$,
the graph $G_t(s_1,\ldots,s_t)$ is $1$-tough.
\end{lemma}

\begin{proof}
Let $G=G_t(s_1,\ldots,s_t)$, and let $K=\{v_1,\ldots,v_t\}$.
Suppose that $G$ is not $1$-tough and choose an inclusion-minimal vertex cut $S\subseteq V(G)$ such that $c(G-S)>|S|$.

First, we claim that $S$ contains no internal vertex of any path $Q_i$.
Indeed, suppose that $z\in S$ is an internal vertex of some $Q_i$.
If $|S|=1$, then $S=\{z\}$; However, deleting one internal vertex of $Q_i$ does not disconnect $G$, since $v_0$ and $v_i$ remain connected through the clique $K$ and the other paths.
Thus $|S|\ge 2$.
Adding $z$ back to $G-S$ can decrease the number of components by at most $1$, and since $c(G-S)>|S|$, we have
\[
   c(G-(S\setminus\{z\}))\ge c(G-S)-1>|S|-1.
\]
In particular, $G-(S\setminus\{z\})$ is disconnected and $c(G-(S\setminus\{z\}))>|S\setminus\{z\}|$, contradicting the minimality of $S$.
Hence $S\subseteq \{v_0,v_1,\ldots,v_t\}$.

If $v_0\notin S$, then $G-S$ is connected, since all remaining vertices are connected through $v_0$ and the clique $K$.
Thus $v_0\in S$.
Let $r=|S\cap K|$.
If $r<t$, then $G-S$ has at most $r+1$ components: one component containing $K\setminus S$, and at most one additional component for each deleted vertex of $K$.
Hence $c(G-S)\le r+1=|S|$.
If $r=t$, then $G-S$ has at most $t$ components, while $|S|=t+1$.
In both cases, $c(G-S)\le |S|$, a contradiction.
Therefore $G$ is $1$-tough.
\end{proof}

\begin{proposition}\label{prop:pointy}
If a graph $G$ has a vertex of degree at least 4 whose neighbors all have degree 2, then $G$ is not RN.
\end{proposition}
\begin{proof}
Suppose that $G$ is RN. Since $G$ is not a path, it is 2-connected by \cite[Proposition 3.7]{Devriendt2025}.
Let $v_0$ be the vertex of degree $t\geq 4$ whose neighbors $N=\{u_1,\dots,u_t\}$ all have degree 2.
Each $u_i$ is then incident to the edge $v_0u_i$ and one other edge which we call $e_i$; these $e_i$ need not be distinct.
Pick $\bx\in P(G)^\circ\cap 2M(G)$.
Then
\begin{align*}
2(|V|-1) = 2(\mathbf 1\cdot \bx)
= \sum_{v\in V} \bx(E(v))
&= \sum_{i=1}^t \bx(E(u_i)) + \sum_{v\in V\setminus N} \bx(E(v))\\
&= \sum_{i=1}^t (\bx_{v_0u_i}+\bx_{e_i}) + \sum_{v\in V\setminus N} \bx(E(v))\\
&= \bx(E(v_0)) + \sum_{i=1}^t \bx_{e_i} + \sum_{v\in V\setminus N} \bx(E(v)).
\end{align*}
Since $\bx\in 2M(G)$, Lemma~\ref{lem:matching-polytope}(ii) gives $\bx(E(v))\leq 2$ for all $v\in V$.
Since $\bx\in P(G)^\circ$ and $G$ is 2-connected, 
Lemma~\ref{lem:treepoly-relint}(iii) gives $\bx_{e_i}<1$ for each $i$.
Continuing from above, we now have
\[ 2|V|-2 < 2 + t + 2(|V|-t) = 2|V|-t+2, \]
a contradiction when $t\geq4$.
\end{proof}

Applying Proposition~\ref{prop:pointy} then yields the following.

\begin{corollary}\label{cor:g4notrn}
For $t\ge 4$, the graph $G_t(s_1,\ldots,s_t)$ is not RN.
\end{corollary}

\onetoughNotRP*

\begin{proof}
Consider the graphs $G_4(s_1,s_2,s_3,s_4)$, which are $1$-tough by Lemma~\ref{lem:Gt-1-tough} but not RN by Corollary~\ref{cor:g4notrn}.
The smallest such graph has $9$ vertices, namely $G_4(1,1,1,1)$, and every larger number of vertices is achieved by incrementing one of the parameters $s_i$.
\end{proof}

\section{Existence of a Nontraceable RP Graph}\label{sec:RP-nontraceable}

In this section, we consider the Thomassen $34$-graph $T_{34}$ (depicted in Figure~\ref{fig:T34}) and prove Theorem~\ref{thm:RP-nontraceable}, restated here for convenience.

\thomassenRPNTraceable*

\begin{figure}[htbp]
  \centering
  %% The Thomassen 34-graph T_34: four Petersen-like wheels, each a 5-star of
%% inner vertices joined by spokes to an outer 5-cycle, wired together by two
%% long edges.  Edge labels a, b, c are the weights witnessing that T_34 is RP;
%% their values are given in the caption.
%%
%% Each \foreach below pairs an index with the label its edge carries, so the
%% label lists read in the same order as the picture is drawn.
%%
%% rotate=90 turns the drawing a quarter turn.  Upright by nature this graph is
%% taller than it is wide, which wasted the page width and capped how large it
%% could be; on its side it is wider than tall (aspect 1.23 against 0.89), so it
%% fits the text block far better.  Node text is not affected by the picture's
%% transformation, so the labels stay upright; the xshift/yshift offsets below
%% do rotate with the drawing, which is what keeps them radial to their wheels.
%%
%% \ttfsize is this figure's size on the page.  The label offsets divide by it,
%% so a label sits the same distance from its edge whatever the drawing size:
%% the label TEXT is a fixed size on the page (nominal times \figscale), while
%% an offset given in pt scales with scale=, so without the division the labels
%% close in on their edges every time the figure is made smaller.
%%
%% The old scale= key was this figure's size on the page: the drawing is then
%% reduced by \figscale, shared with every other figure, so that discs, edges,
%% and labels come out the same size throughout.  Adjust scale= to resize.
\def\ttfsize{1.28441}
\begin{tikzpicture}[rotate=90, scale=\ttfsize]
  % Perpendicular clearance of a label from its edge.
  %
  % Two things about xshift/yshift on a node attached to a path: they are NOT
  % scaled by scale=, and they are NOT turned by rotate=.  They act in canvas
  % space, whatever the picture's transformation.  So these are plain page
  % points (times \figscale) and need no division by \ttfsize -- and every
  % offset angle below carries an extra +90 to undo the picture's rotate=90.
  % Without that the offsets come out parallel to the edges they label, which
  % puts each label exactly on its own edge.
  \pgfmathsetmacro{\ofs}{8}
  \pgfmathsetmacro{\ofsstar}{7}
  \solidnodes
  % Edge labels sit beside the midpoint of their edge, offset by (\dx,\dy) pt.

  % --- spokes of the four wheels, inner vertex to outer vertex ---
  % upper left
  \foreach \x/\lab in {0/b,1/c,2/c,3/b,4/a} {
    \node (ai\x) at ([shift={(-1.5,1.9)}] 360*\x/5:1) {};
    \node (ao\x) at ([shift={(-1.5,1.9)}] 360*\x/5:2) {};
    \pgfmathsetmacro{\dx}{\ofs*cos(72*\x+180)}
    \pgfmathsetmacro{\dy}{\ofs*sin(72*\x+180)}
    \draw[thick] (ai\x) --
      node[textlabel,midway,xshift=\dx pt,yshift=\dy pt] {$\lab$} (ao\x);
  }

  % lower left
  \foreach \x/\lab in {0/b,1/a,2/b,3/c,4/c} {
    \node (bi\x) at ([shift={(-1.5,-1.9)}] 360*\x/5:1) {};
    \node (bo\x) at ([shift={(-1.5,-1.9)}] 360*\x/5:2) {};
    \pgfmathsetmacro{\dx}{\ofs*cos(72*\x+180)}
    \pgfmathsetmacro{\dy}{\ofs*sin(72*\x+180)}
    \draw[thick] (bi\x) --
      node[textlabel,midway,xshift=\dx pt,yshift=\dy pt] {$\lab$} (bo\x);
  }

  % upper right
  \foreach \x/\lab in {0/,1/a,2/b,3/c,4/c} {
    \node (ci\x) at ([shift={(1.5,1.9)}] 360*\x/5+180:1) {};
    \node (co\x) at ([shift={(1.5,1.9)}] 360*\x/5+180:2) {};
    \pgfmathsetmacro{\dx}{\ofs*cos(72*\x+180)}
    \pgfmathsetmacro{\dy}{\ofs*sin(72*\x+180)}
    \draw[thick] (ci\x) --
      node[textlabel,midway,xshift=\dx pt,yshift=\dy pt] {$\lab$} (co\x);
  }

  % lower right
  \foreach \x/\lab in {0/,1/c,2/c,3/b,4/a} {
    \node (di\x) at ([shift={(1.5,-1.9)}] 360*\x/5+180:1) {};
    \node (do\x) at ([shift={(1.5,-1.9)}] 360*\x/5+180:2) {};
    \pgfmathsetmacro{\dx}{\ofs*cos(72*\x+180)}
    \pgfmathsetmacro{\dy}{\ofs*sin(72*\x+180)}
    \draw[thick] (di\x) --
      node[textlabel,midway,xshift=\dx pt,yshift=\dy pt] {$\lab$} (do\x);
  }

  % --- the inner 5-stars ---
  \foreach \x in {ai,bi,ci,di} {
    \draw[thick] (\x0)--(\x2)--(\x4)--(\x1)--(\x3)--(\x0);
  }

  % --- the two edges joining the left wheels to the right ones ---
  \draw[thick] (ao1) --
    node[textlabel,midway,xshift=\ofs pt] {$b$} (co4);
  \draw[thick] (bo4) --
    node[textlabel,midway,xshift=\ofs pt] {$b$} (do1);

  % --- outer arcs of each wheel ---
  % The left wheels offset their labels outwards; the right wheels are rotated
  % by 180 degrees, so theirs are negated.  All four label lists differ.
  %
  % Each entry carries its own sign \sg as a third field.  The four a-labeled
  % arcs are the ones meeting the two vertices at the centre of the picture,
  % where two wheels touch; with the default sign all four land in the narrow
  % notch between the touching wheels, two labels crowding one gap.  Their sign
  % is flipped so each sits inside its own wheel instead.
  \foreach[count=\i from 2] \x/\lab/\sg in {1/c/1,2/b/1,3/a/-1} {
    \pgfmathsetmacro{\dx}{\sg*\ofs*cos(72*\x+126)}
    \pgfmathsetmacro{\dy}{\sg*\ofs*sin(72*\x+126)}
    \draw[thick] (ao\x) --
      node[textlabel,midway,xshift=\dx pt,yshift=\dy pt] {$\lab$} (ao\i);
  }
  \foreach[count=\i from 2] \x/\lab/\sg in {1/a/-1,2/b/1,3/c/1} {
    \pgfmathsetmacro{\dx}{\sg*\ofs*cos(72*\x+126)}
    \pgfmathsetmacro{\dy}{\sg*\ofs*sin(72*\x+126)}
    \draw[thick] (bo\x) --
      node[textlabel,midway,xshift=\dx pt,yshift=\dy pt] {$\lab$} (bo\i);
  }
  \foreach[count=\i from 2] \x/\lab/\sg in {1/a/-1,2/b/1,3/c/1} {
    \pgfmathsetmacro{\dx}{-\sg*\ofs*cos(72*\x+126)}
    \pgfmathsetmacro{\dy}{-\sg*\ofs*sin(72*\x+126)}
    \draw[thick] (co\x) --
      node[textlabel,midway,xshift=\dx pt,yshift=\dy pt] {$\lab$} (co\i);
  }
  \foreach[count=\i from 2] \x/\lab/\sg in {1/c/1,2/b/1,3/a/-1} {
    \pgfmathsetmacro{\dx}{-\sg*\ofs*cos(72*\x+126)}
    \pgfmathsetmacro{\dy}{-\sg*\ofs*sin(72*\x+126)}
    \draw[thick] (do\x) --
      node[textlabel,midway,xshift=\dx pt,yshift=\dy pt] {$\lab$} (do\i);
  }

  % --- labels on the inner stars ---
  \foreach \y in {ai} {
    \foreach[count=\i from 0] \x/\lab in {2/c,3/c,4/b,0/c,1/b} {
      \pgfmathsetmacro{\dx}{\ofsstar*cos(72*\i+162)}
      \pgfmathsetmacro{\dy}{\ofsstar*sin(72*\i+162)}
      \draw[thick] (\y\x) --
        node[textlabel,pos=0.8,xshift=\dx pt,yshift=\dy pt] {$\lab$} (\y\i);
    }
  }
  \foreach \y in {bi} {
    \foreach[count=\i from 0] \x/\lab in {2/c,3/b,4/c,0/c,1/b} {
      \pgfmathsetmacro{\dx}{\ofsstar*cos(72*\i+162)}
      \pgfmathsetmacro{\dy}{\ofsstar*sin(72*\i+162)}
      \draw[thick] (\y\x) --
        node[textlabel,pos=0.8,xshift=\dx pt,yshift=\dy pt] {$\lab$} (\y\i);
    }
  }
  \foreach \y in {ci} {
    \foreach[count=\i from 0] \x/\lab in {2/c,3/b,4/c,0/c,1/b} {
      \pgfmathsetmacro{\dx}{-\ofsstar*cos(72*\i+162)}
      \pgfmathsetmacro{\dy}{-\ofsstar*sin(72*\i+162)}
      \draw[thick] (\y\x) --
        node[textlabel,pos=0.8,xshift=\dx pt,yshift=\dy pt] {$\lab$} (\y\i);
    }
  }
  \foreach \y in {di} {
    \foreach[count=\i from 0] \x/\lab in {2/c,3/c,4/b,0/c,1/b} {
      \pgfmathsetmacro{\dx}{-\ofsstar*cos(72*\i+162)}
      \pgfmathsetmacro{\dy}{-\ofsstar*sin(72*\i+162)}
      \draw[thick] (\y\x) --
        node[textlabel,pos=0.8,xshift=\dx pt,yshift=\dy pt] {$\lab$} (\y\i);
    }
  }
\end{tikzpicture}
  \caption{The graph $T_{34}$, with edge labels $a=\frac{33}{68}$, $b=\frac{99}{136}$, and $c=\frac{165}{272}$.}
  \label{fig:T34}
\end{figure}

\begin{proof}
The graph $T_{34}$ is known to be hypotraceable~\cite{Thomassen1974}, meaning that $T_{34}$ is not traceable, while $T_{34}-v$ is traceable for every $v\in V(T_{34})$.
In particular, $T_{34}$ is $2$-connected.
It remains to show that $T_{34}$ is RP.

Let $\bx\in \bbR^{E(T_{34})}$ be the vector whose coordinates are the edge labels $a,b,c$ shown in Figure~\ref{fig:T34}.
We first show that $\bx\in P(T_{34})^\circ$.
Since $T_{34}$ is $2$-connected, it suffices, by Lemma~\ref{lem:treepoly-relint}, to verify that $\bx_e>0$ for every $e\in E(T_{34})$, that $\bx(E(T_{34}))=|V(T_{34})|-1=33$, and that $\bx(E[S])<|S|-1$ for every $S\subsetneq V(T_{34})$ with $2\le |S|\le 33$.

The first condition is immediate from the definitions of $a,b,c$.
Also, from Figure~\ref{fig:T34}, one checks that $\bx(E(v))=33/17$ for every $v\in V(T_{34})$.
Hence
\[
   2\bx(E(T_{34}))
   =\sum_{v\in V(T_{34})}\bx(E(v))
   =34\cdot \frac{33}{17}
   =66,
\]
so $\bx(E(T_{34}))=33$.

It remains to check the strict induced-subgraph inequalities.
Let $S\subsetneq V(T_{34})$ with $2\le |S|\le 33$, and let $\partial S$ be the set of edges with exactly one endpoint in $S$.
Since $\sum_{v\in S}\bx(E(v))=2\bx(E[S])+\bx(\partial S)$, we have
\[
   2\bx(E[S])=\frac{33}{17}|S|-\bx(\partial S).
\]

We claim that $\bx(\partial S)\ge 33/17$.
If $|V(T_{34})\setminus S|=1$, then $\partial S=E(v)$ for the unique vertex $v\in V(T_{34})\setminus S$, and therefore $\bx(\partial S)=\bx(E(v))=33/17$.
Otherwise, both $S$ and its complement contain at least two vertices.
A computer search verifies that every such cut has at least four edges; that is, $|\partial S|\ge 4$ whenever $2\le |S|\le 32$.
Since every edge has weight at least $a=33/68$, it follows that $\bx(\partial S)\ge 4a=33/17$.
Therefore
\[
   2\bx(E[S])
   \le \frac{33}{17}|S|-\frac{33}{17}
   = \frac{33}{17}(|S|-1),
\]
and hence $\bx(E[S])\le (33/34)(|S|-1)<|S|-1$.
Thus $\bx\in P(T_{34})^\circ$.

Since $\bx\in P(T_{34})^\circ$ and $\bx(E(v))=33/17<2$ for every $v\in V(T_{34})$, Theorem~\ref{thm:devriendt-characterization} shows that $T_{34}$ is RP.
\end{proof}

\section{Sprawling Graphs}\label{sec:sprawling}

We now seek sufficient conditions for RN and RP graphs, with the aim of classifying Cartesian products of paths.
Recall from the introduction that a graph $G=(V,E)$ is called \emph{sprawling} if there exists a collection $\cS=\{H_1,\ldots,H_q\}$ of Hamiltonian paths of $G$ satisfying the following two conditions:
\begin{enumerate}
    \item every edge of $G$ is contained in at least one path in $\cS$;
    \item for every set $U\subsetneq V$ with $2\le |U|\le |V|-1$ such that $G[U]$ is connected, there exists $H_i\in\cS$ such that $H_i[U]$ is not a spanning tree of $G[U]$.
\end{enumerate}
We call such a collection $\cS$ a \emph{sprawling set}.

\begin{example}\label{ex:sprawling}
Figure~\ref{fig:sprawling-examples} shows two examples of sprawling graphs and one non-example.
For the $3$-cycle in Figure~\ref{fig:sprawling-ex-1}, its three Hamiltonian paths form a sprawling set.
The Petersen graph with vertices labeled as in Figure~\ref{fig:sprawling-ex-2} admits the following sprawling set:
\[
\begin{aligned}
H_1&=v_6v_8v_{10}v_7v_9v_4v_5v_1v_2v_3,
&\qquad
H_2&=v_7v_2v_1v_6v_9v_4v_3v_8v_{10}v_5,\\
H_3&=v_8v_6v_9v_7v_{10}v_5v_4v_3v_2v_1,
&\qquad
H_4&=v_6v_1v_5v_4v_9v_7v_{10}v_8v_3v_2.
\end{aligned}
\]
The graph $G_3(1,2,1)$ in Figure~\ref{fig:sprawling-non-ex-1} is traceable but not sprawling: since the edge $e$ is contained in every Hamiltonian path, $U=\{u,v\}$ fails condition (2) above.

\begin{figure}[htbp]
  \centering
  \scalebox{\fpeval{\figscale*1}}{%
  \begin{minipage}{\fpeval{0.81/\figscale}\textwidth}\centering
    %% K_3
\begin{subfigure}[t]{.2593\textwidth}
  \centering
  \begin{tikzpicture}[scale=1.2706]
    \solidnodes
    \node (1) at (0,{sqrt(3)+2}) {};
    \node (2) at (-1,2) {};
    \node (3) at (1,2) {};
    \draw[thick] (1) -- (2) -- (3) -- (1);
  \end{tikzpicture}
  \caption{}
  \label{fig:sprawling-ex-1}
\end{subfigure}\hfill
%% labeled Petersen
\begin{subfigure}[t]{.3345\textwidth}
  \centering
  \begin{tikzpicture}[scale=0.6293]
    \solidnodes
    % vertices, outer ones labeled v_1..v_5 and inner ones v_6..v_10
    \foreach \x in {0,1,2,3,4}{
      \pgfmathtruncatemacro{\olab}{\x+1}
      \pgfmathtruncatemacro{\ilab}{\x+6}
      \node (i\x) at (90+360*\x/5:1.25) {};
      \node[textlabel, font=\normalsize]
        at (112+360*\x/5:1.62) {$v_{\ilab}$};
      \node (o\x) at (90+360*\x/5:3.5) {};
      \node[textlabel, font=\normalsize]
        at (90+360*\x/5:4.05) {$v_{\olab}$};
    }
    % spokes
    \foreach \x in {0,1,2,3,4}{
      \draw[thick] (i\x) -- (o\x);
    }
    % outer pentagon
    \foreach \x/\y in {0/1, 1/2, 2/3, 3/4, 4/0}{
      \draw[thick] (o\x) -- (o\y);
    }
    % inner star
    \foreach \a/\b in {0/2, 2/4, 4/1, 1/3, 3/0}{
      \draw[thick] (i\a) -- (i\b);
    }
  \end{tikzpicture}
  \caption{}
  \label{fig:sprawling-ex-2}
\end{subfigure}\hfill
%% G_3(1,2,1)
\begin{subfigure}[t]{.2593\textwidth}
  \centering
  \begin{tikzpicture}[scale=0.8471]
    \solidnodes
    \node (v) at (3,0) {};
    \foreach \x in {-2,0,2}{
      \node (\x) at (0,\x) {};
    }
    \node (a-2) at (1.5,-1) {};
    \node (a2) at (1.5,1) {};
    \node (b0) at (2,0) {};
    \node (a0) at (1,0) {};

    \node[textlabel, font=\normalsize] at (1,8pt) {$u$};
    \node[textlabel, font=\normalsize] at (2,8pt) {$v$};

    % the three paths to v
    \draw[thick] (2) -- (a2) -- (v);
    \draw[thick] (0) -- (a0) -- (b0) -- (v);
    \draw[thick] (-2) -- (a-2) -- (v);
    % the triangle
    \draw[thick] (2) -- (0) -- (-2);
    \draw[thick] (-2) to[bend left=30] (2);
    % the forced edge, named $e$ to match the text: the label goes below,
    % since $u$ and $v$ already sit above its two endpoints.
    \draw[very thick,red] (a0) --
      node[midway, textlabel, yshift=-10pt] {$e$} (b0);
  \end{tikzpicture}
  \caption{}
  \label{fig:sprawling-non-ex-1}
\end{subfigure}%
  \end{minipage}}
  \caption{Two examples and one non-example of sprawling graphs.}
  \label{fig:sprawling-examples}
\end{figure}
\end{example}

We first prove Theorem~\ref{thm:sprawling-RN}, which we restate here for convenience.

\sprawlingRN*

\begin{proof}
Let $\cS=\{H_1,\ldots,H_q\}$ be a sprawling set for $G$, and define $\bx=\frac{1}{q}\sum_{i=1}^q \be_{H_i}$.
Since each $H_i$ is a Hamiltonian path, we have $\be_{H_i}\in P(G)\cap 2M(G)$, and hence $\bx\in P(G)\cap 2M(G)$.

We claim that $\bx\in P(G)^\circ$.
By condition (1) in the definition of a sprawling set, every edge of $G$ is contained in at least one path in $\cS$.
Therefore $\bx_e>0$ for every $e\in E(G)$ and Lemma~\ref{lem:treepoly-relint}(i) holds.
Also, since each $H_i$ is a spanning tree, $\mathbf 1\cdot \bx=|V(G)|-1$ so Lemma~\ref{lem:treepoly-relint}(ii) holds.

It remains to check Lemma~\ref{lem:treepoly-relint}(iii).
Let $U\subsetneq V(G)$ with $2\le |U|\le |V(G)|-1$.
If $G[U]$ is disconnected, then $H_i[U]$ is disconnected for every $i$, and so $|E(H_i[U])|\le |U|-2$ for every $i$.
Hence $\bx(E[U])<|U|-1$.

Now suppose that $G[U]$ is connected.
By condition (2) in the definition of a sprawling set, there exists $H_j\in\cS$ such that $H_j[U]$ is not a spanning tree of $G[U]$.
Since $H_j[U]$ is a forest on the vertex set $U$, this gives $|E(H_j[U])|\le |U|-2$.
For every other $i$, we have $|E(H_i[U])|\le |U|-1$.
Therefore
\[
   \bx(E[U])
   =
   \frac{1}{q}\sum_{i=1}^q |E(H_i[U])|
   <
   |U|-1.
\]
Thus, $\bx$ satisfies Lemma~\ref{lem:treepoly-relint}, so $\bx\in P(G)^\circ$.
Since $\bx\in 2M(G)$ as well, we have $P(G)^\circ\cap 2M(G)\neq \emptyset$.
By Theorem~\ref{thm:devriendt-characterization}, $G$ is RN.

Finally, by~\cite[Proposition~3.7]{Devriendt2025}, the only RN graphs that are not $2$-connected are paths.
A path with at least three vertices is not sprawling, since its unique Hamiltonian path restricts to a spanning tree on every proper subpath.
Thus, if $|V(G)|\ge 3$, the graph $G$ is $2$-connected.
\end{proof}

A small strengthening of the same argument gives a sufficient condition for RP.

\begin{proposition}\label{prop:sprawling-path-endpoint-RP}
Let $\cS=\{H_1,\ldots,H_q\}$ be a sprawling set of a graph $G=(V,E)$.
Suppose that for every vertex $v\in V$, there exists $H_i\in\cS$ with $v$ as an endpoint.
Then $G$ is RP.
\end{proposition}

\begin{proof}
Let $\bx=\frac{1}{q}\sum_{i=1}^q\be_{H_i}$.
As in the proof of Theorem~\ref{thm:sprawling-RN}, we have $\bx\in P(G)^\circ\cap 2M(G)$.
Moreover, for every $v\in V$, each Hamiltonian path $H_i$ has degree at most $2$ at $v$, and at least one path in the collection has degree $1$ at $v$.
Therefore
\[
   \bx(E(v))
   =
   \frac{1}{q}\sum_{i=1}^q \deg_{H_i}(v)
   <2
\]
for every $v\in V$.
By Theorem~\ref{thm:devriendt-characterization}, $G$ is RP.
\end{proof}

\begin{example}\label{ex:K23}
We remark that the sprawling property is not preserved under adding edges.
Write $K_{2,3}$ with bipartition $\{u,v\},\{a,b,c\}$, and let $K_{2,3}+e$ be the graph obtained from $K_{2,3}$ by adding the edge $uv$.
The graph $K_{2,3}+e$ is not sprawling because the edge $uv$ is contained in no Hamiltonian path.
However, $K_{2,3}$ is sprawling by taking the set of all Hamiltonian paths as the sprawling set.
\end{example}

We now apply the sprawling criterion to Cartesian products of paths.
Let $P_n$ denote the path on $n$ vertices.
We now prove Theorem~\ref{thm:grid-RN}, which we restate here for convenience.

\gridsprawlingRN*

\begin{proof}
Let $G=P_{n_1}\times\cdots\times P_{n_d}$, where $d,n_1,\ldots,n_d$ are positive integers.
If at most one factor has more than one vertex, then $G$ is a path, and hence is RN, as shown by Devriendt~\cite{Devriendt2025}.
Otherwise, after removing all trivial factors $P_1$, we may assume that $d\ge 2$ and $n_i\ge 2$ for every $i\in[d]$.
Each vertex of $G$ is identified with a $d$-tuple $(x_1,\ldots,x_d)$ with $x_i\in [n_i]$.

We recursively construct a family $\cS_G$ of Hamiltonian paths of $G$ such that every edge of $G$ is contained in some path in $\cS_G$.
For a single path $P_n$, let $\cS_{P_n}$ consist of the unique Hamiltonian path of $P_n$.
Now suppose $G=P_{n_1}\times\cdots\times P_{n_d}$ with $d\ge 2$, and assume by induction that the families have already been constructed for products with $d-1$ factors in such a way that every edge is contained in some Hamiltonian path of the corresponding family.

Fix $r\in[d]$, and define $G_{\widehat r}:=\prod_{j\in[d]\setminus\{r\}}P_{n_j}$.
We regard a vertex $z\in V(G_{\widehat r})$ as an indexed tuple $z=(z_j)_{j\in[d]\setminus\{r\}}$.
For $i\in[n_r]$, write $(z;i)_r$ for the vertex $x\in V(G)$ defined by $x_r=i$ and $x_j=z_j$ for every $j\neq r$.
An \emph{$r$-line} of $G$ is a maximal path along which only the $r$th coordinate varies.
Thus the $r$-line over a vertex $z\in V(G_{\widehat r})$ is the path with vertex set $\{(z;i)_r\st i\in[n_r]\}$, whose endpoints are $(z;1)_r$ and $(z;n_r)_r$.

For $H\in\cS_{G_{\widehat r}}$, choose an orientation of $H$, and let $x^{(1)},x^{(2)},\ldots,x^{(m)}$ be the vertices of $H$ in this order.
The superscript indexes the order along $H$.
We define two Hamiltonian paths $H_r^+$ and $H_r^-$ of $G$ as follows.
The path $H_r^+$ traverses the $r$-line over $x^{(1)}$ from $(x^{(1)};1)_r$ to $(x^{(1)};n_r)_r$, then moves to $(x^{(2)};n_r)_r$, traverses the $r$-line over $x^{(2)}$ from $(x^{(2)};n_r)_r$ to $(x^{(2)};1)_r$, and continues in this alternating way until every $r$-line over a vertex of $H$ has been traversed.
The path $H_r^-$ is defined similarly, except that it starts by traversing the $r$-line over $x^{(1)}$ from $(x^{(1)};n_r)_r$ to $(x^{(1)};1)_r$.

Since consecutive vertices of $H$ are adjacent in $G_{\widehat r}$, the end of each traversed $r$-line is adjacent to the beginning of the next one.
Moreover, for this fixed $r$, the $r$-lines partition $V(G)$, and each $r$-line is traversed exactly once.
Hence $H_r^+$ and $H_r^-$ are Hamiltonian paths of $G$.
Now let
\[
   \cS_G
   :=
   \bigcup_{r\in[d]}
   \bigl\{H_r^+,H_r^- \st H\in\cS_{G_{\widehat r}}\bigr\},
\]
where, for each fixed $r$ and $H\in\cS_{G_{\widehat r}}$, the paths $H_r^+$ and $H_r^-$ are the two paths constructed above from $H$.
We illustrate the construction of $\cS_{P_3\times P_3}$ in Figure~\ref{fig:grid}.
\begin{figure}[htbp]
  \centering
  %% The sprawling set S_{P_3 x P_3}: four copies of the 3x3 grid, each carrying
%% one of the Hamiltonian paths H_1^-, H_1^+, H_2^-, H_2^+. 
\begin{tikzpicture}[scale=1.22845,
    edge/.style={black},
    hpath/.style={line width=3pt}
  ]
  \solidnodes

  % P_3 x P_3, drawn at the origin of the enclosing scope
  \newcommand{\GridThree}{%
    \foreach \x in {1,2,3}
    \foreach \y in {1,2,3}
    \node (v\x\y) at (\x,\y) {};
    \foreach \x in {1,2}
    \foreach \y in {1,2,3}
    \draw[edge] (v\x\y) -- (v\the\numexpr\x+1\relax\y);
    \foreach \x in {1,2,3}
    \foreach \y in {1,2}
    \draw[edge] (v\x\y) -- (v\x\the\numexpr\y+1\relax);
  }

  % H_1^- : horizontal, starting top-left
  \begin{scope}[xshift=0cm]
    \GridThree
    \draw[hpath,green!60!black]
    (v13)--(v23)--(v33)--
    (v32)--(v22)--(v12)--
    (v11)--(v21)--(v31);
    \node[textlabel] at (2,0.4) {$H^-_1$};
  \end{scope}

  % H_1^+ : horizontal, starting bottom-right
  \begin{scope}[xshift=4cm]
    \GridThree
    \draw[hpath,red]
    (v33)--(v23)--(v13)--
    (v12)--(v22)--(v32)--
    (v31)--(v21)--(v11);
    \node[textlabel] at (2,0.4) {$H^+_1$};
  \end{scope}

  % H_2^- : vertical, starting top-left
  \begin{scope}[xshift=8cm]
    \GridThree
    \draw[hpath,orange]
    (v13)--(v12)--(v11)--
    (v21)--(v22)--(v23)--
    (v33)--(v32)--(v31);
    \node[textlabel] at (2,0.4) {$H^-_2$};
  \end{scope}

  % H_2^+ : vertical, starting bottom-right
  \begin{scope}[xshift=12cm]
    \GridThree
    \draw[hpath,blue]
    (v11)--(v12)--(v13)--
    (v23)--(v22)--(v21)--
    (v31)--(v32)--(v33);
    \node[textlabel] at (2,0.4) {$H^+_2$};
  \end{scope}
\end{tikzpicture}
  \caption{The set $\cS_{P_3\times P_3}$.
  The vertex identified with $(i,j)$ is represented by the point with coordinate $(i,j)$, with the first coordinate increasing from left to right and the second coordinate increasing from bottom to top.}
  \label{fig:grid}
\end{figure}

We claim that $\cS_G$ is a sprawling set for $G$.
First, $\cS_G$ covers every edge of $G$.
Indeed, let $e=uv$ be an edge of $G$.
Then there is a unique coordinate $r\in[d]$ such that $u_j=v_j$ for all $j\neq r$ and $|u_r-v_r|=1$.
In other words, $e$ lies on an $r$-line of $G$.
For any $H\in\cS_{G_{\widehat r}}$, the path $H_r^+$ traverses every $r$-line completely, and hence contains $e$.

It remains to verify the induced-subgraph condition.
Let $U\subsetneq V(G)$ with $2\le |U|\le |V(G)|-1$, and suppose that $G[U]$ is connected.
Suppose, for a contradiction, that $R[U]$ is a spanning tree of $G[U]$ for every $R\in\cS_G$.
In particular, $R[U]$ is connected for every $R\in\cS_G$.
We use the following elementary fact about paths.
If $R$ is a path and $R[U]$ is connected, then for any two vertices $x,y\in U$, the whole $x$--$y$ subpath of $R$ has all its vertices in $U$.

We now show that $U$ is closed under adjacency.
Let $x\in U$.
Since $G[U]$ is connected and has at least two vertices, there is a vertex $y\in U$ adjacent to $x$.
Let $r$ be the unique coordinate in which $x$ and $y$ differ.
Fix a coordinate $s\neq r$.
Delete the $s$th coordinate from $x$ and $y$, and let $\overline{x}$ and $\overline{y}$ be their images in $G_{\widehat s}$.
Since $x$ and $y$ differ in coordinate $r\neq s$, the vertices $\overline{x}$ and $\overline{y}$ are adjacent in $G_{\widehat s}$.
By the induction hypothesis, some path $H\in\cS_{G_{\widehat s}}$ contains the edge $\overline{x}\,\overline{y}$.
Consider the two paths $H_s^+$ and $H_s^-$ in $\cS_G$ constructed from this path $H$ by extending in coordinate $s$.
In one of these two paths, the subpath from $x$ to $y$ runs through one end of the two $s$-lines through $x$ and $y$; in the other path, the corresponding subpath runs through the other end.
Since both $H_s^+[U]$ and $H_s^-[U]$ are connected, the elementary fact above implies that both subpaths have all their vertices in $U$.
Together these two subpaths contain the full $s$-lines through $x$ and $y$.
Hence, for every $s\neq r$, the full $s$-line through $x$ is contained in $U$.
It remains to see that the full $r$-line through $x$ is also contained in $U$.
Choose some coordinate $s\neq r$.
Since $n_s\ge 2$, the full $s$-line through $x$ contains a neighbor $w\in U$ of $x$ in coordinate direction $s$.
Applying the argument above to the edge $xw$, with coordinate $r$ in place of $s$, shows that the full $r$-line through $x$ is contained in $U$.

Thus every coordinate line through the arbitrary vertex $x\in U$ is contained in $U$.
In particular, every neighbor of $x$ in $G$ lies in $U$.
Since $x$ was arbitrary, $U$ is closed under adjacency.
Because $G$ is connected and $U$ is nonempty, this implies $U=V(G)$, contradicting that $U$ is proper.

Hence there exists some $R\in\cS_G$ such that $R[U]$ is not a spanning tree of $G[U]$.
Thus $\cS_G$ is a sprawling set, and so $G$ is sprawling.
By Theorem~\ref{thm:sprawling-RN}, the graph $G$ is RN.
\end{proof}

\begin{corollary}\label{cor:path-product-RP-classification}
Let $d\ge 2$ and $n_1,\ldots,n_d\ge 2$.
Then $P_{n_1}\times\cdots\times P_{n_d}$ is RP if and only if at least one $n_i$ is even.
If all $n_i$ are odd, then the product is SRN.
\end{corollary}

\begin{proof}
Let $G=P_{n_1}\times\cdots\times P_{n_d}$.
If at least one $n_i$ is even, then $G$ is Hamiltonian, and hence is RP by~\cite[Theorem~6.2]{Devriendt2025}.
Now suppose that $n_1,\ldots,n_d$ are all odd.
By Theorem~\ref{thm:grid-RN}, the graph $G$ is RN.
For each $i\in[d]$, let $(A_i,B_i)$ be the bipartition of $P_{n_i}$ chosen so that $|A_i|=(n_i+1)/2$ and $|B_i|=(n_i-1)/2$.
Let $X$ be the bipartition class of $G$ consisting of the vertices $(v_1,\ldots,v_d)$ for which $v_i\in B_i$ for an even number of indices $i$, and let $Y$ be the other bipartition class.
For each $S\subseteq[d]$, the number of vertices for which $v_i\in B_i$ exactly when $i\in S$ is $\prod_{i\in S}|B_i|\prod_{i\notin S}|A_i|$.
Therefore
\[
   |X|-|Y|
   =
   \sum_{S\subseteq[d]}(-1)^{|S|}
   \prod_{i\in S}|B_i|
   \prod_{i\notin S}|A_i|
   =
   \prod_{i=1}^d\bigl(|A_i|-|B_i|\bigr)
   =
   1.
\]
Thus the two bipartition classes of $G$ have unequal sizes.
By~\cite[Proposition~3.5]{Devriendt2025}, the graph $G$ is not RP.
Hence $G$ is RN but not RP, and is therefore SRN.
\end{proof}

\appendix
\section{Comparison of graph properties}\label{sec:zoo}
Throughout this paper, we have encountered various levels of connectedness (connected, 2-connected, 1-tough), resistance curvature (RN, RP), and Hamiltonicity (traceable, sprawling, Hamiltonian). 
This appendix presents the state-of-the-art for implications among these graph properties. 
See Figure~\ref{fig:containment} for a pictorial summary. All graphs in this appendix and Figure~\ref{fig:containment} are assumed to be connected and have at least three vertices.
Each cell in the figure holds the smallest known (to the authors) example with that cell's combination of properties.
Scripts to verify the claimed properties of each graph can be found in the associated GitHub repository\footnote{\href{https://github.com/Gyagr7/Effective-Resistance-Curvature-Graph-Properties}{https://github.com/Gyagr7/Effective-Resistance-Curvature-Graph-Properties}}, which includes the \texttt{rn} utility used to produce a rational certificate point for RP or RN.
The certificates presented here are symmetrized over graph automorphisms for expository ease.
Where the example is smallest possible, $n$ is its number of vertices; otherwise the cell gives the upper bound.
The cells corresponding to $1$-tough SRN graphs are absent; \cite[Theorem~B]{garcia2026} demonstrates that no such graph exists.

\begin{figure}[p]
  \centering
  \resizebox{\textwidth}{!}{\input{figures/containment-page.tikz}}
  \caption{Containment diagram of the properties discussed, with the smallest known witness in each cell; $n$ is the least order a graph in that cell can have.}
  \label{fig:containment}
\end{figure}

Several examples are from the family $G_t(s_1,\dots,s_t)$ constructed in Section~\ref{sec:1-tough-RP}.
While we established many of the needed properties there, we start by proving an additional lemma that will be applied. 

\begin{lemma}\label{lem:pointypaths}
  Suppose $G$ has a vertex $v_{0}$ of degree $t\ge 3$ whose neighbors all have degree $2$.
  Then $v_{0}$ is interior to every Hamiltonian path.
  Consequently, $G$ is not traceable if $t\ge 5$.
\end{lemma}
\begin{proof}
  Let $u_{1},\ldots,u_{t}$ be the neighbors of $v_{0}$.
  Take a Hamiltonian path $H$.
  Towards a contradiction, suppose that $v_{0}$ is one of the endpoints of $H$ so that $H$ begins with the edge $v_0u_j$ for some $j$.
  Each remaining $u_i$ has degree 1 in $G-v_0$, so each is an endpoint of $H$.
  There are $t-1\ge 2$ such vertices and one endpoint remaining, a contradiction.
  Thus, $v_0$ is not an endpoint of $H$.

  Moreover, if $t\geq 5$ and $v_0$ is internal to a Hamiltonian path $H$, then there are $t-2\geq3$ vertices $u_i$ which must be endpoints of $H$. No such $H$ could exist.
\end{proof}

\begin{example}\label{ex:G5}
The graphs $G_4(1,1,1,1)$ and $G_5(1,1,1,1,1)$ are 1-tough by Lemma~\ref{lem:Gt-1-tough} and not RN by Corollary~\ref{cor:g4notrn}.
While $G_4(1,1,1,1)$ is traceable, the graph $G_5(1,1,1,1,1)$ is not traceable by Lemma~\ref{lem:pointypaths}.
\end{example}

\begin{example}\label{ex:G3}
The graphs $G_3(1,1,1)$ and $G_3(1,2,1)$ are both 1-tough by Lemma~\ref{lem:Gt-1-tough}. 
Each is traceable but not Hamiltonian.
Example~\ref{ex:sprawling} shows that $G_{3}(1,2,1)$ is not sprawling.
In contrast, the twelve Hamiltonian paths of $G_{3}(1,1,1)$ form a sprawling set:
\[ \cS = \{v_{i}u_{i}v_{0}u_{j}v_{j}v_{k}u_{k},\ u_{i}v_{0}u_{j}v_{j}v_{i}v_{k}u_{k} \st \{i,j,k\}=\{1,2,3\}\}, \]
where the vertices are labeled as in the construction of $G_{3}(s_{1},s_2,s_{3})$ in Section~\ref{sec:1-tough-RP}.

Despite this difference, both graphs are RP.
The edge weights in Figure~\ref{fig:g3-111} correspond to a point $\bx \in P(G_{3}(1,1,1))^{\circ}\cap 2M(G_{3}(1,1,1))$ with $\bx(E(v))<2$ at every vertex, so $G_{3}(1,1,1)$ is RP.
Similarly, the point $\by$ corresponding to Figure~\ref{fig:g3-121} in the same way certifies that $G_{3}(1,2,1)$ is RP.

\begin{figure}[htbp]
  \centering
  \captionsetup[sub]{justification=centering,singlelinecheck=true}
  \begin{subfigure}[b]{0.48\textwidth}
    \centering
    %% $G_3(1,1,1)$, RP, with the symmetrised witness of `src/rn --exact'
\begin{tikzpicture}[
    wvtx/.style={circle, fill=black, inner sep=0pt, minimum size=4pt},
    wlab/.style={font=\tiny, inner sep=0.6pt},
    wedge/.style={line width=0.4pt, black!75}]
  \coordinate (v) at (0.000,0.000);
  \coordinate (x1) at (0.000,1.364);
  \coordinate (x2) at (-1.181,-0.682);
  \coordinate (x3) at (1.181,-0.682);
  \coordinate (y1) at (0.000,2.970);
  \coordinate (y2) at (-2.572,-1.485);
  \coordinate (y3) at (2.572,-1.485);
  \draw[wedge] (v) -- (x1);
  \draw[wedge] (x1) -- (y1);
  \draw[wedge] (v) -- (x2);
  \draw[wedge] (x2) -- (y2);
  \draw[wedge] (v) -- (x3);
  \draw[wedge] (x3) -- (y3);
  \draw[wedge] (y1) -- (y2);
  \draw[wedge] (y2) -- (y3);
  \draw[wedge] (y3) -- (y1);
  \node[wlab, fill=white] at (0.000,0.682) {$12/19$};
  \node[wlab, fill=white] at (0.000,2.167) {$17/19$};
  \node[wlab, fill=white] at (-0.591,-0.341) {$12/19$};
  \node[wlab, fill=white] at (-1.877,-1.084) {$17/19$};
  \node[wlab, fill=white] at (0.591,-0.341) {$12/19$};
  \node[wlab, fill=white] at (1.877,-1.084) {$17/19$};
  \node[wlab, fill=white] at (-1.286,0.743) {$9/19$};
  \node[wlab, fill=white] at (0.000,-1.485) {$9/19$};
  \node[wlab, fill=white] at (1.286,0.742) {$9/19$};
  \node[wvtx] at (v) {};
  \node[wvtx] at (x1) {};
  \node[wvtx] at (x2) {};
  \node[wvtx] at (x3) {};
  \node[wvtx] at (y1) {};
  \node[wvtx] at (y2) {};
  \node[wvtx] at (y3) {};
\end{tikzpicture}
    \caption{$G_{3}(1,1,1)$}
    \label{fig:g3-111}
  \end{subfigure}
  \hfill
  \begin{subfigure}[b]{0.48\textwidth}
    \centering
    %%  rn-G3-121.tikz --- $G_3(1,2,1)$, RP, with the symmetrised witness of `src/rn --exact'
\begin{tikzpicture}[
    wvtx/.style={circle, fill=black, inner sep=0pt, minimum size=4pt},
    wlab/.style={font=\tiny, inner sep=0.6pt},
    wedge/.style={line width=0.4pt, black!75}]
  \coordinate (m2) at (0.000,1.056);
  \coordinate (v) at (0.000,0.000);
  \coordinate (x1) at (1.181,-0.682);
  \coordinate (x2) at (0.000,2.046);
  \coordinate (x3) at (-1.181,-0.682);
  \coordinate (y1) at (2.572,-1.485);
  \coordinate (y2) at (0.000,2.970);
  \coordinate (y3) at (-2.572,-1.485);
  \draw[wedge] (v) -- (x1);
  \draw[wedge] (x1) -- (y1);
  \draw[wedge] (v) -- (m2);
  \draw[wedge] (m2) -- (x2);
  \draw[wedge] (x2) -- (y2);
  \draw[wedge] (v) -- (x3);
  \draw[wedge] (x3) -- (y3);
  \draw[wedge] (y1) -- (y2);
  \draw[wedge] (y2) -- (y3);
  \draw[wedge] (y3) -- (y1);
  \node[wlab, fill=white] at (0.591,-0.341) {$7/11$};
  \node[wlab, fill=white] at (1.877,-1.084) {$10/11$};
  \node[wlab, fill=white] at (0.000,0.528) {$7/11$};
  \node[wlab, fill=white] at (0.000,1.551) {$10/11$};
  \node[wlab, fill=white] at (0.000,2.508) {$10/11$};
  \node[wlab, fill=white] at (-0.591,-0.341) {$7/11$};
  \node[wlab, fill=white] at (-1.877,-1.084) {$10/11$};
  \node[wlab, fill=white] at (1.286,0.742) {$1/2$};
  \node[wlab, fill=white] at (-1.286,0.743) {$1/2$};
  \node[wlab, fill=white] at (0.000,-1.485) {$5/11$};
  \node[wvtx] at (m2) {};
  \node[wvtx] at (v) {};
  \node[wvtx] at (x1) {};
  \node[wvtx] at (x2) {};
  \node[wvtx] at (x3) {};
  \node[wvtx] at (y1) {};
  \node[wvtx] at (y2) {};
  \node[wvtx] at (y3) {};
\end{tikzpicture}
    \caption{$G_{3}(1,2,1)$}
    \label{fig:g3-121}
  \end{subfigure}
  \caption{Edge weights certifying that both $G_{3}(1,1,1)$ and $G_{3}(1,2,1)$ are RP.}
  \label{fig:g3pair}
\end{figure}
\end{example}

We now leave the family $G_t(s_1,\dots,s_t)$ and build on Example~\ref{ex:K23} for the next two examples.

\begin{example}\label{ex:K23-appendix}
The complete bipartite graph $K_{2,3}$ is sprawling by Example~\ref{ex:K23}, hence RN by Theorem~\ref{thm:sprawling-RN}.
Since $K_{2,3}$ is bipartite with unequal parts,
it is not RP by~\cite[Proposition~3.5]{Devriendt2025}. 
\end{example}

\begin{example}\label{ex:K23+e}
As in Example~\ref{ex:K23}, write $K_{2,3}$ with bipartition $\{u,v\}\cup\{a,b,c\}$, and let $K_{2,3}+e$ be the graph obtained by adding edge $uv$ to $K_{2,3}$.
The graph $K_{2,3}+e$ is traceable; we show here that it is not RN.
Suppose for contradiction that $K_{2,3}+e$ is RN and select $\bx\in P(K_{2,3}+e)^{\circ}\cap 2M(K_{2,3}+e)$.
Every edge of $K_{2,3}+e$ meets $\{u,v\}$, and $e$ meets it twice,
so $\bx(E(u))+\bx(E(v))=\bx(E)+\bx_e$.
The left-hand side is at most 4 by Lemma~\ref{lem:matching-polytope}(ii), while the right-hand side is greater than 4 by Lemma~\ref{lem:treepoly-relint}(i)--(ii), yielding the desired contradiction.
\end{example}

Finally, in these last two examples, we present  the graphs $X_{37}$ and $X_{29}$, both of which were found by computer search, and justify their claimed properties.
Both graphs are constructed using the 11-vertex bipartite graph $\Gamma$ in Figure~\ref{fig:cert37}.
Note that $\Gamma$ is not traceable.
Its two vertices of degree $1$ must be the ends of any Hamiltonian path, so every other vertex is interior to it.
The five vertices of degree $2$ therefore put both of their incident edges on the path, and six of these eight edges so forced make a cycle of length $6$, which no path contains.

\begin{figure}[htbp]
  \centering
  \providecolor{vermilion}{HTML}{D55E00}
\providecolor{okabeblue}{HTML}{0072B2}
\begin{tikzpicture}[scale=1.30,
    dot/.style={circle, fill=black, inner sep=0pt, minimum size=4.5pt},
    fedge/.style={line width=0.7pt},
    xlab/.style={font=\scriptsize, inner sep=1pt, fill=white,
                 fill opacity=1.00, text opacity=1},
    dlab/.style={font=\scriptsize\itshape, inner sep=1.5pt, text=vermilion}]
  \coordinate (0) at (2.35,0);
  \coordinate (1) at (0,0);
  \coordinate (2) at (4.7,0);
  \coordinate (3) at (7.05,0);
  \coordinate (4) at (11.75,0);
  \coordinate (5) at (9.4,0);
  \coordinate (6) at (1.175,3.3);
  \coordinate (7) at (3.525,3.3);
  \coordinate (8) at (10.575,3.3);
  \coordinate (9) at (5.875,3.3);
  \coordinate (10) at (8.225,3.3);
  \draw[fedge] (0)--(6);
  \draw[fedge] (6)--(1);
  \draw[fedge] (1)--(9);
  \draw[fedge] (0)--(7);
  \draw[fedge] (7)--(2);
  \draw[fedge] (2)--(9);
  \draw[fedge] (0)--(10);
  \draw[fedge] (10)--(3);
  \draw[fedge] (3)--(9);
  \draw[fedge] (10)--(5);
  \draw[fedge] (3)--(8);
  \draw[fedge] (8)--(4);

  \node[dot] at (0) {};
  \node[dot] at (1) {};
  \node[dot] at (2) {};
  \node[dot] at (3) {};
  \node[dot] at (4) {};
  \node[dot] at (5) {};
  \node[dot] at (6) {};
  \node[dot] at (7) {};
  \node[dot] at (8) {};
  \node[dot] at (9) {};
  \node[dot] at (10) {};

  \node[xlab] at (1.780,1.600) {$32/45$};
  \node[xlab] at (0.570,1.600) {$44/45$};
  \node[xlab] at (1.406,0.790) {$4/5$};
  \node[xlab] at (2.631,0.790) {$32/45$};
  \node[xlab] at (4.130,1.600) {$44/45$};
  \node[xlab] at (4.981,0.790) {$4/5$};
  \node[xlab] at (3.756,0.790) {$22/45$};
  \node[xlab] at (7.620,1.600) {$8/15$};
  \node[xlab] at (6.480,1.600) {$2/5$};
  \node[xlab] at (9.119,0.790) {$44/45$};
  \node[xlab] at (7.894,0.790) {$44/45$};
  \node[xlab] at (11.180,1.600) {$44/45$};

  \node[dlab, anchor=north] at (2.35,-0.155) {$1/45$};
  \node[dlab, anchor=north] at (0,-0.155) {$1/18$};
  \node[dlab, anchor=north] at (4.7,-0.155) {$1/18$};
  \node[dlab, anchor=north] at (7.05,-0.155) {$1/45$};
  \node[dlab, anchor=north] at (11.75,-0.155) {$23/90$};
  \node[dlab, anchor=north] at (9.4,-0.155) {$23/90$};
\end{tikzpicture}
  \caption{The graph $\Gamma$, carrying the edge weights of the certificate for $X_{37}$.}
  \label{fig:cert37}
\end{figure}

\begin{example}\label{ex:X37}
Construct a bipartite graph $X_{37}$ as follows.
In our bipartite graph $\Gamma$, let $A$ be the larger part with 6 vertices, and let $B$ be the smaller part with 5 vertices.
Take three copies of $\Gamma$, naming each $\Gamma_{i}$ with parts $A_{i},B_{i}$ for $1\le i \le 3$.
Let $X_{37}$ be the disjoint union of these $\Gamma_i$ with four additional vertices $y_1,y_2,y_3,y_4$, where each $y_j$ is adjacent to every vertex in the union $A_1\cup A_2\cup A_3$.
The graph $X_{37}$ is thus bipartite on the parts $A_1\cup A_2\cup A_3$ of size $18$ and $\{y_{1},y_{2},y_{3},y_{4}\}\cup B_1\cup B_2\cup B_3$ of size $19$.
As such, $X_{37}$ is not 1-tough.
We will show here that $X_{37}$ is SRN and not traceable.

Towards a contradiction, suppose that $X_{37}$ is traceable and has a Hamiltonian path $H$.
Removing the four vertices $y_{j}$ splits $H$ into five subpaths, some possibly empty, so that, up to reordering the $y_{j}$, the path $H$ reads $Q_{1}y_{1}Q_{2}y_{2}Q_{3}y_{3}Q_{4}y_{4}Q_{5}$.
Since no edge joins two distinct copies $\Gamma_{i}$, each $Q_{j}$ lies entirely within a single copy.
By the pigeonhole principle, five subpaths distributed across three copies leave some copy $\Gamma_{i}$ containing at most one of them.
That copy contains at least one, since its vertices lie on $H$, so the single subpath meets every vertex of $\Gamma_{i}$ and is therefore a Hamiltonian path of $\Gamma_{i}$, a contradiction.
Thus, $X_{37}$ is not traceable.

Now construct edge weights $\bx$ on $X_{37}$.
On each copy $\Gamma_i$, the weight $\bx$ is the one drawn on the edges of Figure~\ref{fig:cert37}.
The remaining edges are those joining a vertex of some $A_{i}$ to some $y_{j}$, and each takes the weight labeling its endpoint in $A_{i}$ from the same figure, written in vermilion. 
One verifies, by computer or by hand, that $\bx\in P(X_{37})^{\circ}\cap 2M(X_{37})$, so $X_{37}$ is RN.
Since it is bipartite with unequal parts, it is not RP by~\cite[Proposition~3.5]{Devriendt2025} and not 1-tough.
In Figure~\ref{fig:containment}, the graph $X_{37}$ thus occupies the cell of graphs that are $2$-connected, SRN, and not traceable.
\end{example}

\begin{example}\label{ex:X29}
Construct $X_{29}$ in the same way as $X_{37}$, except that one copy of $\Gamma$ is replaced by $K_{1,2}$.
As before, $X_{29}$ is bipartite with unequal parts, hence not 1-tough. We will show here that $X_{29}$ is traceable, not sprawling, and RN.

The graph $X_{29}$ is traceable.
Take the decomposition $Q_{1}y_{1}Q_{2}y_{2}Q_{3}y_{3}Q_{4}y_{4}Q_{5}$ as before, with $Q_{3}$ the copy of $K_{1,2}$.
Each copy $\Gamma_i$ supplies two of the subpaths, the single edge and the $9$-vertex path highlighted in Figure~\ref{fig:cert29-gamma}: let $Q_{1}$ and $Q_{5}$ be the two single edges, one from each copy, and let $Q_{2}$ and $Q_{4}$ be the two $9$-vertex paths.
Together the five subpaths meet every vertex of $X_{29}$.

Apart from the endpoints of $Q_1$ and $Q_5$ in the copies $B_i$, every end of a subpath lies in some $A_{i}$, so the vertices $y_{j}$ join the subpaths in any order.
More generally, in any Hamiltonian path $H$ each of the two copies of $\Gamma$ must hold at least two of the five subpaths, since a single subpath covering a copy would be a Hamiltonian path of $\Gamma$.
Therefore the copy of $K_{1,2}$ is met in exactly one subpath, which covers all three of its vertices and so uses both of its edges.
As this is true of every Hamiltonian path, the second condition for sprawling fails at the vertex set $U$ of that copy of $K_{1,2}$.

Read as before, the edge weights and vertex labels of Figures~\ref{fig:cert29-gamma} and~\ref{fig:cert29-k12} give an RN certificate for $X_{29}$.
Therefore, in Figure~\ref{fig:containment}, the graph $X_{29}$ occupies the cell of graphs that are $2$-connected, SRN, traceable, and not sprawling.

\begin{figure}[htbp]
  \centering
  \captionsetup[sub]{justification=centering,singlelinecheck=true}
  \renewcommand{\figscale}{0.76}
  \begin{subfigure}[b]{0.76\textwidth}
    \centering
    \providecolor{vermilion}{HTML}{D55E00}
\providecolor{okabeblue}{HTML}{0072B2}
\begin{tikzpicture}[scale=1.30,
    dot/.style={circle, fill=black, inner sep=0pt, minimum size=4.5pt},
    fedge/.style={line width=0.7pt},
    xlab/.style={font=\scriptsize, inner sep=1pt, fill=white,
                 fill opacity=1.00, text opacity=1},
    dlab/.style={font=\scriptsize\itshape, inner sep=1.5pt, text=vermilion}]
  % X_29, same geometry, X_29's numbers
  \coordinate (0) at (2.35,0);
  \coordinate (1) at (0,0);
  \coordinate (2) at (4.7,0);
  \coordinate (3) at (7.05,0);
  \coordinate (4) at (11.75,0);
  \coordinate (5) at (9.4,0);
  \coordinate (6) at (1.175,3.3);
  \coordinate (7) at (3.525,3.3);
  \coordinate (8) at (10.575,3.3);
  \coordinate (9) at (5.875,3.3);
  \coordinate (10) at (8.225,3.3);
  \tikzset{
    hpath/.style={line width=3.8pt, draw=okabeblue!62, line cap=round},
    hpathalt/.style={hpath}}
  \draw[hpath] (0)--(6);
  \draw[hpath] (6)--(1);
  \draw[hpath] (0)--(7);
  \draw[hpath] (7)--(2);
  \draw[hpath] (2)--(9);
  \draw[hpath] (10)--(3);
  \draw[hpath] (3)--(9);
  \draw[hpath] (10)--(5);
  \draw[hpathalt] (8)--(4);

  \draw[fedge] (0)--(6);
  \draw[fedge] (6)--(1);
  \draw[fedge] (1)--(9);
  \draw[fedge] (0)--(7);
  \draw[fedge] (7)--(2);
  \draw[fedge] (2)--(9);
  \draw[fedge] (0)--(10);
  \draw[fedge] (10)--(3);
  \draw[fedge] (3)--(9);
  \draw[fedge] (10)--(5);
  \draw[fedge] (3)--(8);
  \draw[fedge] (8)--(4);

  \node[dot] at (0) {};
  \node[dot] at (1) {};
  \node[dot] at (2) {};
  \node[dot] at (3) {};
  \node[dot] at (4) {};
  \node[dot] at (5) {};
  \node[dot] at (6) {};
  \node[dot] at (7) {};
  \node[dot] at (8) {};
  \node[dot] at (9) {};
  \node[dot] at (10) {};

  \node[xlab] at (1.780,1.600) {$25/38$};
  \node[xlab] at (0.570,1.600) {$18/19$};
  \node[xlab] at (1.406,0.790) {$16/19$};
  \node[xlab] at (2.631,0.790) {$25/38$};
  \node[xlab] at (4.130,1.600) {$18/19$};
  \node[xlab] at (4.981,0.790) {$16/19$};
  \node[xlab] at (3.756,0.790) {$9/19$};
  \node[xlab] at (7.620,1.600) {$11/19$};
  \node[xlab] at (6.480,1.600) {$5/19$};
  \node[xlab] at (9.119,0.790) {$18/19$};
  \node[xlab] at (7.894,0.790) {$18/19$};
  \node[xlab] at (11.180,1.600) {$18/19$};

  \node[dlab, anchor=north] at (2.35,-0.155) {$1/19$};
  \node[dlab, anchor=north] at (0,-0.155) {$1/19$};
  \node[dlab, anchor=north] at (4.7,-0.155) {$1/19$};
  \node[dlab, anchor=north] at (7.05,-0.155) {$1/19$};
  \node[dlab, anchor=north] at (11.75,-0.155) {$5/19$};
  \node[dlab, anchor=north] at (9.4,-0.155) {$5/19$};
\end{tikzpicture}
    \caption{$\Gamma$}
    \label{fig:cert29-gamma}
  \end{subfigure}
  \hfill
  \begin{subfigure}[b]{0.20\textwidth}
    \centering
    \providecolor{vermilion}{HTML}{D55E00}
\providecolor{okabeblue}{HTML}{0072B2}
\begin{tikzpicture}[scale=1.30,
    dot/.style={circle, fill=black, inner sep=0pt, minimum size=4.5pt},
    fedge/.style={line width=0.7pt},
    xlab/.style={font=\scriptsize, inner sep=1pt, fill=white,
                 fill opacity=1.00, text opacity=1},
    dlab/.style={font=\scriptsize\itshape, inner sep=1.5pt, text=vermilion}]
  \coordinate (a) at (0,0);
  \coordinate (ap) at (2.35,0);
  \coordinate (b) at (1.175,3.3);
  \draw[fedge] (a)--(b);
  \draw[fedge] (ap)--(b);
  \node[dot] at (a) {};
  \node[dot] at (ap) {};
  \node[dot] at (b) {};
  \node[xlab] at (0.570,1.600) {$18/19$};
  \node[xlab] at (1.780,1.600) {$18/19$};
  \node[dlab, anchor=north] at (0,-0.155) {$5/19$};
  \node[dlab, anchor=north] at (2.35,-0.155) {$5/19$};
\end{tikzpicture}
    \caption{$K_{1,2}$}
    \label{fig:cert29-k12}
  \end{subfigure}
  \caption{The graphs $\Gamma$ and $K_{1,2}$, carrying the edge weights of the certificate for $X_{29}$; highlighted on $\Gamma$ are the two subpaths that a Hamiltonian path of $X_{29}$ leaves inside one copy of $\Gamma$.}
  \label{fig:cert29}
\end{figure}
\end{example}

\section*{Acknowledgments}
The research was primarily conducted during a Vertically Integrated Research course in the Department of Mathematics at Louisiana State University, affiliated with the Research Training Grant NSF DMS-2231492.
The authors thank Kataria Stewart and Alexander Paul Torres for early discussions on the topic, and thank Xiaonan Liu for her insights and contributions to the proof of Theorem~\ref{thm:1-tough-not-RP}.
Agrahari and Wang's research is supported in part by LA Board of Regents grant LEQSF(2024-27)-RD-A-16, and Bibby's research is supported by NSF DMS-2204299.
OpenAI's ChatGPT and Anthropic's Claude assisted in proofreading the exposition and prose.
Additionally, Claude assisted in creating and editing figures, writing the verification program used in Appendix~\ref{sec:zoo}, and writing the program to produce the RN and RP witnesses in Theorem~\ref{thm:RP-nontraceable}, Example~\ref{ex:X29}, and Example~\ref{ex:X37}.

\printbibliography

\end{document}